\documentclass[a4paper,11pt]{article}

\usepackage{epsfig}
\usepackage{amsmath,amsthm}
\usepackage{amsfonts}
\usepackage{latexsym}
\usepackage{amsbsy}
\usepackage{amssymb}
\usepackage{xypic}
\numberwithin{equation}{section}

\usepackage{amsfonts,amssymb,amscd,amsmath,enumerate,verbatim,calc}
\usepackage{xy}

\newcommand{\uu}{\underline}

\newcommand{\ki}{\Omega^ {\mathcal K}}

\newcommand{\z}{\mathcal B(\mathcal H)}

\newcommand{\m}{\mathcal}
\theoremstyle{plain}
\newtheorem{theorem}{Theorem}[section]

\newtheorem{proposition}[theorem]{Proposition}

\newtheorem{definition}[theorem]{Definition}

\begin{document}

\title{\small \textbf{OUTGOING CUNTZ SCATTERING SYSTEM FOR A COISOMETRIC LIFTING AND TRANSFER FUNCTION}}
\author {\sc Kalpesh J. Haria}
\date{}
\maketitle

\begin{center}
 \noindent Department of Mathematics,\\
Indian Institute of Technology Bombay,\\
Powai, Mumbai- 400076, India\\
E-mail: kalpesh@math.iitb.ac.in
\end{center}

\vspace{1cm}

\begin{abstract}
We study a coisometry that intertwines  Popescu's presentations of
minimal isometric dilations of a given operator tuple and of a coisometric lifting of the tuple. 
Using this we develop an
outgoing Cuntz scattering system which gives rise to an input-output formalism. A transfer
function is introduced for the system.
 We also compare the transfer function
and the characteristic function for the associated lifting.
\end{abstract}

\vspace{1cm}



\section{\textbf Introduction}

The model of repeated interaction between quantum systems has been recently studied
in \cite{Go10} and an outgoing Cuntz scattering
system was associated to the model. In \cite{DH11} the authors gave a vast generalization of the
Gohm's repeated interaction model using the theory of liftings
of row contractions.  We recall that in the generalized
repeated interaction model of \cite{DH11} we have the
following operator theoretic data:

Let  $\m H$ and $ \m K$ be  complex
separable Hilbert spaces such that  $ \tilde{\m H}$ be a
 subspace of $\m H$ and $\ki$ be a distinguished unit
vector of $\m K$. Let $\m P$ be
a  $d$-dimensional Hilbert space with an orthonormal
 basis $\{\epsilon_j\}_{j=1}^d$, and
 $ U : \m H \otimes \m K \to \m H \otimes \m P$ and
$\tilde U :\tilde {\m H} \otimes \m K \to \tilde {\m H} \otimes \m P$
 be two unitaries such that
\begin{align}\label{1}
  U(\tilde h \otimes \ki) =  \tilde U(\tilde h \otimes \ki) ~\mbox{for each}~~ \tilde h \in \tilde {\m H}.
\end{align}
The unitary $ U : \m H \otimes \m K \to \m H \otimes \m P$ can be decomposed as
\[
  U(h \otimes \ki) = \displaystyle \sum_{j=1}^d E_j^* h \otimes \epsilon_j~  \mbox{for each}~~ h\in \m H,
\]
where $E_j$'s are some operators in $\z$, for j = 1, . . . , d. Likewise there exist some operators
 ̃
$C_j$'s in $\m B( \tilde {\m H})$  such that
\[
 \tilde  U(\tilde h \otimes \ki) = \displaystyle \sum_{j=1}^d C_j^*\tilde h \otimes \epsilon_j~  \mbox{for each}~~ \tilde
 h\in \tilde{\m H}.
\]
Observe that
$\ \sum_{j=1}^d E_jE_j^*= I$
 and
$\ \sum_{j=1}^d C_jC_j^*= I$, i.e., $\uu E$ and $\uu C$ are coisometric tuples.
Then equation (\ref{1}) yields that  $E_j^*\tilde h =
 C_j^*\tilde h$ for each $\tilde h \in \tilde {\m H}$, i.e.,   the tuple $\uu E$
  is a \textit{lifting} (cf. \cite{DG11}) of the tuple $\uu C$.
In this article  we proceed in the opposite way, viz. we start with coisometric
 $d$-tuples $\uu E = (E_1,\ldots, E_d)$ and $\uu C = (C_1,\ldots, C_d)$
such that the tuple $\uu E$ is a lifting of the tuple
$\uu C$, then associate a pair of unitaries to it but satisfying similar relations
as  above and then study the model.

The article is organized as follows: After associating a pair of unitaries to a given coisometric lifting
as stated above, we identify a coisometry which intertwines corresponding  minimal isometric dilations in section 2.
We investigate the properties of this coisometry and express 
it as a limit of certain compositions of these unitaries.
 The Cuntz scattering system was introduced in
\cite{BV05} using the generalization of
Lax-Phillips scattering system to a multivariate operator setting.
In section 3
we study the forward part of a Cuntz scattering system which is called as an outgoing Cuntz
scattering system in \cite{BV05}. Let  $\tilde {\Lambda}$ denote the
free semi group with generators $1,\ldots, d$.
Using an input-output formalism we define a
colligation of operators \cite{BV05} which gives rise to a $\tilde \Lambda$-linear system $\sum_{U, \tilde U}$.
 An application of the generalized Fourier transform
to the $\tilde \Lambda$-linear system $\sum_{U, \tilde U}$ under zero initial condition
leads to the input-output relation
\[
 \hat y(z) = \Theta_{U, \tilde U} \hat u(z)
\]
between the Fourier transforms of input and output variables
where $\Theta_{U, \tilde U}$ is the transfer function of the system.
These transfer functions are multi-analytic operators.
There are other approaches to transfer functions in
\cite{YK03} and \cite{GGY08}.

Popescu
introduced
the characteristic function in \cite{Po89b}
of a row contraction, and systematically developed an extensive theory of row contractions (cf.
\cite{Po95}, \cite{Po99}). We use some of the concepts from Popescu's theory in this work.
In \cite{DG11} Dey and Gohm described a class of multi-analytic operators which classify
certain class of liftings and called them characteristic functions for liftings.
 We find a relation between the transfer function
and the characteristic function for the associated lifting.

The following \textit{multi-index notation} will be used frequently in this article.
 Suppose $T_1, \ldots ,T_d \in \m B(\m L)$ for a Hilbert space $\m L.$
If $\alpha \in \tilde \Lambda $ is a word $\alpha_1\ldots\alpha_n$
 with length  $|\alpha |= n$ where each $\alpha_j \in \{1,\ldots,d\},$
then $T_\alpha$  denotes  $T_{\alpha_1}\ldots T_{\alpha_n}.$
For the empty word $\emptyset$
we define $|\emptyset| = 0$ and $T_\emptyset = I$.
The full Fock space over $\mathbb C^d $ denoted by
$\Gamma$ is the Hilbert space
\[
  \Gamma  := \mathbb{C}\oplus\mathbb{C}^d\oplus(\mathbb{C}^d)^{\otimes^2}
\oplus \ldots\oplus(\mathbb{C}^d)^{\otimes^m}\oplus \ldots .
\]
The element $1\oplus 0 \oplus \ldots$ of  $\Gamma$ is
called the vacuum vector. Let $\{e_1,\ldots, e_d\}$ be the
standard orthonormal basis of $\mathbb C^d$. For $\alpha = \alpha_1\ldots\alpha_m\in \tilde \Lambda$, $e_\alpha $
will denote the vector $e_{\alpha_1}\otimes\ldots \otimes e_{\alpha_m}$
in the full Fock space $\Gamma $ and $e_\emptyset$ will denote the vacuum vector.
$\{e_\alpha: \alpha \in \tilde \Lambda\}$ forms an orthonormal basis of the
full Fock space.

\section{Unitaries Associated to a Coisometric Lifting}
In this section we construct a coisometry associated to a coisometric lifting which plays
an important role in this article. For simplicity we assume that $d$ is
 finite throughout this article but all the results here can be derived also for $d = \infty$.
A tuple $\uu T = (T_1,\ldots, T_d)$
of bounded linear operators on a
Hilbert space $\m L$ is said to be a \textit{row contraction} if $\sum_{j=1}^{d}T_jT_j^* \leq I$. In particular
if $\sum_{j=1}^{d}T_jT_j^* = I$, then the tuple  $\uu T = (T_1,\ldots, T_d)$  is called a \textit{coisometric}. If
$T_j$'s are isometries with orthogonal ranges, then the tuple $\uu T = (T_1,\ldots, T_d)$ 
is called a \textit{row isometry}.
Consider a row contraction $\uu T$ as a row operator from $\bigoplus _{1}^{d} \m L $
to $\m L$. Define $D_T := (I - \uu T^*\uu T)^\frac{1}{2} : \bigoplus _{1}^{d}
\m L  \to \bigoplus _{1}^{d} \m L $
and let $\m D_T := \overline {\mbox{Range} ~ D_T}$.
The (left) creation operators $L_j$'s on $\Gamma $ are defined by
$L_j x = e_j \otimes  x$ for $1 \leq j \leq  d$ and $x \in \Gamma  $. The tuple $\uu L = (L_1, \ldots , L_d ) $ is a row
isometry. Popescu in \cite{Po89a} gave the following explicit presentation
 of the minimal isometric dilation of $\uu T$ by $\uu V$ on $\hat {\m L}= \m L \oplus (\Gamma \otimes \m D_T):$
\begin{equation}\label{2}
 V_j(\ell \oplus \displaystyle \sum_{\alpha \in \tilde{\Lambda}} e_\alpha \otimes d_\alpha)
= T_j\ell \oplus [e_\emptyset \otimes (D_T)_j\ell + e_j
\otimes  \displaystyle \sum_{\alpha \in \tilde{\Lambda}} e_\alpha \otimes d_\alpha]
 \end{equation}
where $\ell \in \mathcal L$, $d_\alpha \in \mathcal D_T$, and $(D_T)_j: \m L \to \bigoplus_{1}^{d} \m L$
is defined for $j= 1,\ldots,d$ by the $(D_T)_j \ell =
D_T(0,\ldots, \ell, \ldots, 0)$ with $\ell$ is
embedded at the $j^{th}$ component. If in addition  $\uu T$ is
 the coisometric tuple, then $\uu V$ is also the coisometric tuple.

Let $\uu C  = (C_1,\ldots, C_d)$ be a coisometric tuple on
a Hilbert space $\m H_C$,  $\uu E = (E_1,\ldots, E_d)$ be
a coisometric lifting of $\uu C = (C_1,\ldots, C_d)$ on a Hilbert space $\m H_E \supset \m H_C$ and
\begin{equation}\label{a}
 E_j =
\begin{pmatrix}
          C_j & 0 \\
        B_j & A_j\\
         \end{pmatrix}
\end{equation}
for $j=1,\ldots,d$ with respect to the decomposition $\m H_E = \m
H_C \oplus \m H_C^\perp$. From now on we denote $ \m H_C^\perp $ by $ \m H_A$.
Let $ \underline {\widehat  V}^E = (\widehat V_1^E,\ldots
,\widehat V_d^E)$ and $  \uu{ \widehat  V}^C = (\widehat V_1^C,\ldots
,\widehat V_d^C)$ be minimal isometric dilations of the form given by equation (\ref{2})
for tuples $\uu E$ and $\uu C$
on Hilbert spaces $\widehat {\m H}_E  = \m H_E \oplus (\Gamma \otimes \m D_E)$
  and $\widehat {\m H}_C = \m H_C \oplus (\Gamma \otimes \m D_C)$ respectively.
Let $\m P$ be a $d$-dimensional Hilbert space with an
orthonormal basis $\{\epsilon_1, \ldots, \epsilon_d\}$.
We  define operators  $\tilde U : \m H_C \oplus
\m D_C \to \m H_C \otimes \m P$ and  $ U : \m H_E \oplus
\m D_E \to \m H_E \otimes \m P$ as follows:
\[
  \tilde U(\tilde h \oplus y) :=  \sum_{j = 1}^d (C_j^*\tilde h + (D_C)_j^* y) \otimes \epsilon_j
= \sum_{j = 1}^d (\widehat V_j^C)^* (\tilde h \oplus e_\emptyset \otimes y)
  \otimes \epsilon_j,
\]
\[
U( h \oplus \eta) := \sum_{j = 1}^d (E_j^* h + (D_E)_j^* \eta) \otimes \epsilon_j
= \sum_{j = 1}^d (\widehat V_j^E)^* ( h \oplus e_\emptyset \otimes \eta)  \otimes \epsilon_j
\]
where $ \tilde h \in \m H_C,~ y \in  \m D_C,~  h \in \m H_E,  ~\mbox{and}~  \eta \in \m D_E $.

We show that $\tilde U$ and $U$ are unitaries.
 From the fact that $ \underline {\widehat  V}^C = (\widehat V_1^C,\ldots
,\widehat V_d^C)$ is a coisometric tuple, it follows that $\tilde U$ is an isometry. We claim that
 $\tilde U$
is a surjective map. For $ \sum_{j = 1}^d \tilde h_j \otimes \epsilon_j \in \m H_C \otimes \m P$
\begin{eqnarray*}
 \tilde U(\sum_{j = 1}^d(C_j\tilde h_j \oplus (D_C)_j \tilde h_j) &=&
  \sum_{j = 1}^d  \tilde U (C_j\tilde h_j \oplus (D_C)_j \tilde h_j)\\
&=& \sum_{j = 1}^d\big(\sum_{i = 1}^d(C_i^*C_j\tilde h_j + (D_C)_i^*(D_C)_j \tilde h_j) \otimes \epsilon_i\big)\\
&=& \sum_{j = 1}^d \tilde h_j \otimes \epsilon_j
\end{eqnarray*}
where the last equality uses
\begin{eqnarray*}(D_C)_i^*(D_C)_j = \left\{\begin{array}{ll}
-C_i^*C_j ~~~~~~~~~~~~~~\mbox{if}~~  i \neq j;\\
I-C_j^*C_j ~~~~~~~~~~~ \mbox{if}~~i = j. \end{array}\right.
\end{eqnarray*}
Thus $\tilde U$ is unitary. Similarly it can be shown that $U$ is also unitary.

For the main theorem of this section, we need to introduce some operators.
Let us define $\tilde U_1 : \widehat {\m H}_C  \to \m H_C  \otimes \m P \oplus \big(
\displaystyle\bigoplus_{m \geq 1} ((\mathbb C^d)^{\otimes^ m} \otimes \m D_C)\big)$
 by the formula
\begin{align}\label{3}
\tilde U_1 (\tilde h \oplus  \displaystyle \sum_{\alpha \in \tilde{\Lambda}} e_\alpha \otimes y_\alpha) =
\sum_{j = 1}^d \big((C_j^*\tilde h + (D_C)_j^* y_\emptyset)\otimes \epsilon_j\big) \oplus \displaystyle
 \sum_{ |\alpha|\geq 1}e_\alpha \otimes y_\alpha
\end{align}
where $\tilde h \in \m H_C $ and $y_\alpha \in \m D_C$.
For each $n \geq 2$, let
$$\tilde U_n: \m H_C \otimes \m P_{[1, n-1]}\oplus   \big(
\displaystyle\bigoplus_{m \geq n-1} ((\mathbb C^d)^{\otimes^ m} \otimes \m D_C)\big) \to
\m H_C \otimes \m P_{[1, n]}\oplus   \big( \displaystyle\bigoplus_{m \geq n} ((\mathbb C^d)^{\otimes^ m}
 \otimes \m D_C)\big)$$
be defined for  $\tilde h_{j_1,\ldots,j_{n-1}} \in \m H_C$  and $y_{j_1 \ldots j_{n-1}},~y_\alpha \in \m D_C$ by

\begin{multline}\label{4}
\tilde U_n \bigg(\displaystyle \sum_{j_1,\ldots,j_{n-1} = 1}^d(\tilde h_{j_1\ldots j_{n-1}} \otimes
\epsilon_{j_1}\otimes \ldots \otimes \epsilon_{j_{n-1}}) \\
\oplus  \displaystyle \sum_{j_1,
\ldots,j_{n-1} = 1}^d (e_{j_1} \otimes \ldots \otimes e_{j_{n-1}}\otimes y_{j_1 \ldots j_{n-1}})
\oplus \displaystyle \sum_{ |\alpha|\geq n} e_\alpha \otimes y_\alpha\bigg) \\
=\displaystyle \sum_{j_1,\ldots ,j_n = 1}^d \big((C_{j_n}^*\tilde h_{j_1\ldots j_{n-1}}
+ (D_C)_{j_n}^*y_{j_1 \ldots j_{n-1}}) \otimes  \epsilon_{j_1}\otimes \ldots \otimes
 \epsilon_{j_n} \big)\oplus \displaystyle \sum_{ |\alpha|\geq n} e_\alpha \otimes y_\alpha.
\end{multline}
It can be seen that $\tilde U_n$ is unitary for each  $n\geq 1$, with the suitable modifications of the
arguments which are used in proving $\tilde U$ is unitary.

Similarly, we define the unitary operator $ U_1 : \widehat {\m H}_E \to \m H_E \otimes \m P \oplus \big(
\displaystyle\bigoplus_{m \geq 1} ((\mathbb C^d)^{\otimes^ m} \otimes \m D_E)\big)
$ by the following formula
\begin{align}\label{5}
 U_1 ( h \oplus  \displaystyle \sum_{\alpha \in \tilde{\Lambda}} e_\alpha \otimes \eta_\alpha)
 = \sum_{j = 1}^d \big((E_j^* h + (D_E)_j^* \eta_\emptyset)\otimes \epsilon_j\big) \oplus \displaystyle
 \sum_{ |\alpha|\geq 1} e_\alpha \otimes \eta_\alpha
 \end{align}
where $ h \in \m H_E $ and $\eta_\alpha \in \m D_E$. For each $n \geq 2$, define the unitary
$$ U_n:  \m H_E \otimes \m P_{[1, n-1]}\oplus   \big(
\displaystyle\bigoplus_{m \geq n-1} ((\mathbb C^d)^{\otimes^ m} \otimes \m D_E)\big) \to \m H_E
 \otimes \m P_{[1, n]}\oplus   \big( \displaystyle\bigoplus_{m \geq n} ((\mathbb C^d)^{\otimes^ m} \otimes \m D_E)\big)
$$
for $ h_{j_1\ldots j_{n-1}} \in \m H_E$  and $\eta_{j_1 \ldots j_{n-1}},~\eta_\alpha \in \m D_E$ by
\begin{multline}
 U_n \bigg(\displaystyle \sum_{j_1,\ldots,j_{n-1} = 1}^d(h_{j_1\ldots j_{n-1}} \otimes
\epsilon_{j_1}\otimes \ldots \otimes \epsilon_{j_{n-1}})\\
 \oplus \displaystyle
 \sum_{j_1,\ldots,j_{n-1} = 1}^d (e_{j_1} \otimes \ldots \otimes e_{j_{n-1}}
\otimes \eta_{j_1 \ldots j_{n-1}})
 \oplus \displaystyle \sum_{|\alpha|\geq n} e_\alpha \otimes \eta_\alpha\bigg)\nonumber \\
\end{multline}
\begin{multline}\label{6}
 = \displaystyle \sum_{j_1,\ldots ,j_n = 1}^d \big((E_{j_n}^* h_{j_1\ldots j_{n-1}}
+ (D_E)_{j_n}^*\eta_{j_1 \ldots j_{n-1}}) \otimes  \epsilon_{j_1}\otimes \ldots
\otimes \epsilon_{j_n} \big)\oplus \displaystyle \sum_{  |\alpha|\geq n} e_\alpha \otimes \eta_\alpha.
\end{multline}

Since $\uu E$ is a lifting of $\uu C$, $E_j^* \tilde h = C_j^* \tilde h$
 for each $\tilde h \in \m H_C$. It follows from equations (\ref{3}), (\ref {4}), (\ref{5}), and (\ref{6}) that
\begin{equation}\label{7}
\tilde U_1 \tilde h = U_1 \tilde h~~\mbox{and}~~
 \tilde U_n(\tilde h \otimes \epsilon_{j_1}\otimes \ldots \otimes \epsilon_{j_{n-1}})
 = U_n(\tilde h \otimes \epsilon_{j_1}\otimes \ldots \otimes \epsilon_{j_{n-1}})
\end{equation}
for each $\tilde h \in \m H_C $, $1\leq j_1, \ldots ,j_{n-1} \leq d$, and $n\geq 2$.
For each $n\geq 1$, let $Q_n$ denote the orthogonal projection of
$\widehat {\m H}_C$ on to
 $\m H_C \oplus \big(
\displaystyle\bigoplus_{ m \leq n-1} ((\mathbb C^d)^{\otimes^ m} \otimes \m D_C)\big)$, and let $P_n$
denote the orthogonal projection of
 $\m H_E \otimes \m P_{[1, n]}\oplus   \big( \displaystyle \bigoplus_{m \geq n}
 ((\mathbb C^d)^{\otimes^ m} \otimes \m D_E)\big)$ on to $\m H_C \otimes \m P_{[1, n]}$.

We are now ready to prove the main result of this section.

\begin{theorem}
 If $P_n$ and $Q_n$ are the orthogonal projections as defined above for each $n \geq 1$, then
\[
 sot-\displaystyle\lim_{n\to\infty} \tilde U_1^*\ldots \tilde U_n^*P_n U_n\ldots U_1.
\]
exists. This limit is a coisometry, say $\widehat W :\widehat {\m H}_E \to \widehat {\m H}_C$.
 Its adjoint $\widehat W^* :  \widehat {\m H}_C \to \widehat {\m H}_E$ is given by
\[
sot-\displaystyle\lim_{n\to\infty} U_1^*\ldots U_n^*\tilde U_n\ldots\tilde U_1 Q_n
\]
Here sot denotes for the strong operator topology.
\end{theorem}
\begin{proof}
Let us begin by the dense linear manifold $\displaystyle \bigcup_{l \geq 1} \bigg(\m H_C \oplus
\big(\displaystyle\bigoplus_{ m \leq l -1} ((\mathbb C^d)^{\otimes^ m} \otimes
\m D_C)\big)\bigg)$ of $\widehat {\m H}_C$. Assume $\tilde h \in \tilde {\m H}$ and $y_\alpha \in \m D_C$
for all $\alpha \in \tilde \Lambda$ such that $|\alpha|\leq k-1$,
for some positive integer $k$.
We show that
 \[
 \displaystyle\lim_{n\to\infty} U_1^*\ldots U_n^*\tilde U_n\ldots\tilde U_1 Q_n \Big
(\tilde h \oplus \sum_ { |\alpha|\leq k-1} e_\alpha \otimes y_\alpha\Big)
 \]
exists.
For each $n \geq 1$,
set $s_n = U_1^*\ldots U_n^*\tilde U_n\ldots\tilde U_1 Q_n \Big (\tilde h \oplus
 \displaystyle\sum_{ |\alpha|\leq k-1}
e_\alpha \otimes y_\alpha\Big)$. It follows from equation (\ref{7}) that
$s_k = s_{k + j}$ for each $j \geq 1$. Thus
 \[
  \displaystyle\lim_{n\to\infty} U_1^*\ldots U_n^*\tilde U_n\ldots\tilde U_1 Q_n \Big
(\tilde h \oplus \sum_{|\alpha|\leq k-1}
 e_\alpha \otimes y_\alpha\Big)
 \]
exists and it is $s_k$.
We observe that
\begin{eqnarray*}
&&\| \displaystyle\lim_{n\to\infty} U_1^*\ldots U_n^*\tilde U_n\ldots\tilde U_1 Q_n \Big
(\tilde h \oplus \sum_{  |\alpha|\leq k-1}
 e_\alpha \otimes y_\alpha\Big) = \|s_k\|\\
& =& \|U_1^*\ldots U_k^*\tilde U_k\ldots
\tilde U_1 Q_k \Big (\tilde h \oplus \displaystyle\sum_{ |\alpha|\leq k-1} e_\alpha \otimes y_\alpha\Big)\|\\
\nonumber\\
\nonumber &=& \|U_1^*\ldots U_k^*\tilde U_k\ldots\tilde U_1 \Big
(\tilde h \oplus \displaystyle\sum_{|\alpha|\leq k-1}
 e_\alpha \otimes y_\alpha\Big)\|\\
& = &\|\tilde h \oplus \displaystyle\sum_{|\alpha|\leq k-1}
 e_\alpha \otimes y_\alpha \|.\\
\end{eqnarray*}
Hence $sot-\displaystyle\lim_{n\to\infty} U_1^*\ldots U_n^*\tilde U_n\ldots\tilde U_1 Q_n$
 defines an isometry on the dense linear manifold $\displaystyle \bigcup_{l\geq 1}
\bigg(\m H_C \oplus  \big(\displaystyle\bigoplus_{m \leq l -1}
((\mathbb C^d)^{\otimes^ m} \otimes \m D_C)\big)\bigg)$  of $ \widehat {\m H}_C$.
By continuity, it extends to an isometry $\widehat R$ from $ \widehat {\m H}_C$ to $ \widehat {\m H}_E$.
Its adjoint $\widehat R^*: \widehat {\m H}_E \to \widehat {\m H}_C$ is a coisomerty and so just need to rename
$\widehat R^*$ as $\widehat W$.

Let $ h \in \m H_E$ and $\eta_\alpha \in \m D_E$  for all $\alpha\in \tilde \Lambda$ 
 such that $|\alpha|\leq k-1$, for some positive integer $k$.
Let $n \geq k$, $\tilde  h \in \m H_C$  and $y_\beta \in \m D_C$ for all $\beta \in \tilde \Lambda,$
 $|\beta|\leq n-1$. Then
\begin{eqnarray*}
&&  \big\langle \widehat W  \big(h \oplus \displaystyle\sum_{|\alpha|\leq k-1} e_\alpha \otimes \eta_\alpha\big) , ~\tilde
  h \oplus \displaystyle\sum_{|\beta|\leq n-1}
 e_\beta \otimes y_\beta \big\rangle \\
  &=&  \big \langle  h \oplus \displaystyle\sum_{ |\alpha|\leq k-1} e_\alpha \otimes \eta_\alpha ,
~ U_1^*\ldots U_n^*\tilde U_n\ldots\tilde
 U_1 Q_n\big( \tilde  h \oplus \displaystyle\sum_{|\beta|\leq n-1} e_\beta \otimes y_\beta \big)\big\rangle \\
 &=&\big \langle  \tilde U_1^* \ldots \tilde U_n^*
P_n U_n\ldots  U_1 \big(h \oplus \displaystyle\sum_{ |\alpha|\leq k -1} e_\alpha \otimes \eta_\alpha \big),
~  \tilde  h \oplus \displaystyle\sum_{|\beta|\leq n-1} e_\beta \otimes y_\beta \big\rangle. \\
 \end{eqnarray*}
It follows from the above calculations that
\[
Q_n \widehat W\big(h \oplus \displaystyle\sum_{ |\alpha|\leq k-1} e_\alpha \otimes \eta_\alpha\big) = \tilde U_1^*
\ldots \tilde U_n^*  P_n U_n\ldots U_1
 \big(h \oplus \displaystyle\sum_{|\alpha|\leq k-1} e_\alpha \otimes \eta_\alpha \big).
\]
By the fact that $sot- \displaystyle\lim_{n \to\infty} Q_n = I,$ we conclude that
\[
 \widehat W\big(h \oplus \displaystyle\sum_{|\alpha|\leq k-1} e_\alpha \otimes \eta_\alpha\big)
= \displaystyle\lim_{n \to\infty} \tilde U_1^* \ldots \tilde U_n^*
 P_n U_n \ldots U_1 \big(h \oplus \displaystyle\sum_{|\alpha|\leq k-1} e_\alpha \otimes \eta_\alpha \big).
\]
Finally, we extend this formula to the whole of $\widehat {\m H}_E$ by continuity.
\end{proof}

The following result shows that the coisometry $\widehat W$ and its adjoint $\widehat W^*$
intertwine the tuples $\uu {\widehat V}^E$ and $\uu {\widehat V}^C$.

\begin{proposition}\label{2.2}
For $ j= 1,\ldots ,d$, $\widehat W \widehat V_j^E = \widehat V_j^C \widehat W $  and
  $~\widehat V_j^E \widehat W^* = \widehat W^* \widehat V_j^C$.
\end{proposition}
\begin{proof}
Let $ h \in \m H_E$ and $\eta_\alpha \in \m D_E$ for all $\alpha \in \tilde \Lambda$ such that
 $\sum_{\alpha \in \tilde \Lambda} \|\eta_\alpha\|^2 < \infty$.
Suppose $\tilde  h \in \m H_C$  and $y_\beta \in \m D_C$ for all  $\beta \in \tilde \Lambda$
 such that $|\beta|\leq n-1$, for some positive integer $n$. Then
\begin{eqnarray*}
  && \big\langle \widehat W \widehat V_j^E \big( h\oplus \displaystyle
\sum_{\alpha \in \tilde \Lambda} e_\alpha \otimes \eta_\alpha\big),
~ \tilde h\oplus  \displaystyle\sum_{|\beta|\leq n-1}  e_\beta \otimes y_\beta \big \rangle\\
  & = &  \big\langle \widehat V_j^E  \big( h\oplus \displaystyle \sum_{\alpha \in \tilde \Lambda}
 e_\alpha \otimes \eta_\alpha\big),
~ \widehat W^* \big( \tilde h\oplus
\displaystyle\sum_{|\beta|\leq n-1} e_\beta \otimes y_\beta  \big)\big\rangle\\
  & = &  \big\langle E_j h + e_\emptyset \otimes (D_E)_j h +
e_j \otimes  \displaystyle \sum_{\alpha \in \tilde \Lambda}
e_\alpha \otimes \eta_\alpha ,
~ U_1^*\ldots U_n^*\tilde U_n\ldots\tilde U_1
\big( \tilde h\oplus  \displaystyle\sum_{
 |\beta|\leq n-1}  e_\beta \otimes y_\beta  \big)\big\rangle\\
& = &  \big\langle U_1^*\big(h \otimes \epsilon_j \oplus e_j \otimes
 \displaystyle \sum_{\alpha \in \tilde \Lambda} e_\alpha \otimes
\eta_\alpha\big ), ~U_1^*\ldots U_n^*\tilde U_n\ldots\tilde U_2
 \big( \displaystyle \sum_{ i= 1}^d ((C_i^* \tilde h +
(D_C)_i^*y_\emptyset) \otimes  \epsilon_i) \\
&& \oplus  \displaystyle\sum_{ 1\leq|\beta|\leq n-1}
 e_\beta \otimes y_\beta \big)\big\rangle\\
 \end{eqnarray*}
\begin{eqnarray*}
 & = &  \big\langle \big(h \otimes \epsilon_j \oplus e_j \otimes  \displaystyle \sum_{\alpha \in \tilde
 \Lambda} e_\alpha \otimes \eta_\alpha\big ), ~ U_2^*\ldots U_n^*\tilde U_n\ldots\tilde U_2
 \big( \displaystyle \sum_{ i= 1}^d ((C_i^* \tilde h + (D_C)_i^*y_\emptyset) \otimes  \epsilon_i)
 \\
&& \oplus \displaystyle\sum_{1\leq |\beta|\leq n-1}
  e_\beta \otimes y_\beta \big)\big\rangle\\
& = &  \big\langle  h\oplus \displaystyle \sum_{\alpha \in \tilde \Lambda} e_\alpha \otimes \eta_\alpha,
~U_1^*\ldots U_{n-1}^*\tilde U_{n-1} \ldots \tilde U_1 \big( C_j^* \tilde h + (D_C)_j^*y_\emptyset
\oplus L_j^* \otimes I  (\displaystyle\sum_{ \ 1 \leq|\beta|\leq n-1}  e_\beta \otimes y_\beta)\big)\big\rangle\\
   & = &  \big\langle  h\oplus \displaystyle \sum_{\alpha \in \tilde \Lambda} e_\alpha \otimes \eta_\alpha,
~ \widehat  W^*(\widehat V_j^C)^*\big(\tilde h\oplus  \displaystyle\sum_{|\beta|\leq n-1}
 e_\beta \otimes y_\beta\big) \big \rangle\\
   &=& \big\langle \widehat V_j^C \widehat W \big( h\oplus \displaystyle
\sum_{\alpha \in \tilde \Lambda} e_\alpha \otimes \eta_\alpha\big),
 ~ \tilde h\oplus  \displaystyle\sum_{
|\beta|\leq n-1} e_\beta \otimes y_\beta\big \rangle.\\
 \end{eqnarray*}
Thus
 $\widehat W\widehat V_j^E = \widehat V_j^C \widehat W$ for $ j= 1,\ldots ,d.$

\begin{eqnarray*}
  \widehat W^* \widehat V_j^C \big (\tilde h\oplus  \displaystyle\sum_{|\alpha|\leq n-1} e_\alpha  \otimes y_\alpha\big)
&=& \widehat W^* \tilde  U_1^*\big (\tilde h \otimes \epsilon_j \oplus e_j \otimes \displaystyle\sum_{
|\alpha|\leq n-1} e_\alpha  \otimes y_\alpha \big)\\
&=& U_1^*\ldots U_{n+1}^*\tilde U_{n+1}\ldots \tilde U_1 \tilde  U_1^*\big(\tilde h \otimes
\epsilon_j  \oplus e_j \otimes \displaystyle\sum_{|\alpha|\leq n-1} e_\alpha  \otimes y_\alpha \big)\\
&=&  U_1^* U_2^*\ldots U_{n+1}^*\tilde U_{n+1}\ldots \tilde U_2 \big(\tilde h \otimes \epsilon_j
\oplus e_j \otimes \displaystyle\sum_{ |\alpha|\leq n-1} e_\alpha  \otimes y_\alpha \big)\\
&=& \widehat V_j^E \widehat W^* \big (\tilde h\oplus  \displaystyle\sum_{ |\alpha|\leq n-1} e_\alpha \otimes y_\alpha \big).
 \end{eqnarray*}
By continuity, this extends to all of $\widehat{\m H}_C$.
So $\widehat V_j^E \widehat W^* = \widehat W^* \widehat V_j^C$ for $ j= 1,\ldots ,d.$
\end{proof}

We have $(\widehat V_j^E)^* \tilde h = E_j^* \tilde h =  C_j^* \tilde h $ and
  $(\widehat V_j^C)^* \tilde h = C_j^*\tilde h$
 for each  $\tilde h \in \m H_C$,
 i.e., $\m H_C$ is covariant under $\widehat V_j^E$ and $\widehat V_j^C$. Thus $\widehat V_j^E (\m H_A \oplus
(\Gamma \otimes \m D_E)) \subset \m H_A \oplus (\Gamma \otimes \m D_E)$ and
$\widehat V_j^C(\Gamma \otimes \m D_C) \subset \Gamma \otimes \m D_C$.
 Define $V_j^E := \widehat V_j^E|_{\m H_A \oplus (\Gamma \otimes \m D_E)} :
 \m H_A \oplus (\Gamma \otimes \m D_E) \to \m H_A \oplus (\Gamma \otimes \m D_E)$ and
$V_j^C := \widehat V_j^C|_{ \Gamma \otimes \m D_C} :
 \Gamma \otimes \m D_C \to  \Gamma \otimes \m D_C$. In fact, $V_j^C = L_j \otimes I_{\m D_C}$.
Further note that $\widehat W \tilde h = \tilde h$ and $\widehat W^*\tilde h = \tilde h$
for each $\tilde h \in \m H_C$. Define
$$W^* := \widehat W^*|_{ \Gamma \otimes \m D_C} : \Gamma \otimes \m D_C \to \m H_A \oplus (\Gamma \otimes \m D_E). $$
It can be seen that $W$, the adjoint of $W^*$, is given by
$\widehat W|_{ \m H_A \oplus (\Gamma \otimes \m D_E)}:
 \m H_A \oplus (\Gamma \otimes \m D_E) \to \Gamma \otimes \m D_C.$
Then it follows from Proposition \ref{2.2}  that
\begin{align*}
  W V_j^E =  V_j^C  W
~\mbox{and}~
 V_j^E  W^* =  W^* V_j^C ~ \mbox{for} ~ j= 1,\ldots ,d.
\end{align*}

\section{Outgoing Cuntz Scattering System, Transfer Function and Characteristic Function of Lifting}
In order to define an outgoing Cuntz scattering system,
we need the following:


\begin{definition}\rm
\begin{itemize}\item [(1)]  Let  $\uu T = (T_1 ,\ldots, T_d)$ be a row isometry on a Hilbert space $\m L$.
A subspace $\m M$ of $\m L$ is called  \textit{wandering subspace} with respect to
$\uu T$ if
\[
 T_\alpha \m M \perp T_\beta \m M ~\mbox{for distinct}~ \alpha , \beta \in \tilde \Lambda.
\]
\item [(2)]  A  tuple $\uu T = (T_1 ,\ldots, T_d)$ on a Hilbert space $\m L$ is called a
 \textit{row unitary} if $\uu T$ is a row isometry and
$ \overline{span}_{j=1,\ldots ,d} T_j \m L = \m L $.
\item [(3)]  A tuple $\uu T = (T_1, \ldots, T_d)$ on a Hilbert space $\m L$  is called \textit{row shift}
 if $\uu T$ is a row isometry and there exits a wandering
 subspace $\m M$  of $\m L$ with respect to $\uu T$ such that
 $\m L = \oplus_{\alpha \in \tilde {\Lambda}} T_\alpha \m M$.
\end{itemize}
\end{definition}

We omit the proofs of the Theorem \ref{3.7} and Theorem \ref{b} in this section because they
follow using similar arguments as those in
sections 4 and 5 of \cite{DH11}.
In  Chapter 5 of \cite {BV05} an outgoing Cuntz scattering system is defined as a  collection
$$(\uu V = (V_1,\ldots, V_d),~ \m L,~\m G^+_*,~\m G) $$ such that
$\uu V$ is a row isometry on the Hilbert space $\m L$, and
$\m G^+_*$ and $\m G$ are subspaces of $\m L$ such that
\begin{itemize}
\item [(a)] $\m G_*^+$ is the smallest $\uu V-$invariant subspace containing
\[
 \m E_* := \m L \ominus \overline{span}_{j =1,\ldots , d} V_j \m L ;
\]
thus $\uu V|_{\m G_*^+}$ is a row shift and
$\m G_{*}^+ = \bigoplus _{\alpha \in \tilde \Lambda}  V_\alpha \m E_*$.
\item [(b)] $\uu V|_{\m G}$ is a row shift; thus
$\m G =  \bigoplus _{\alpha \in \tilde \Lambda}  V_\alpha \m E$
where $\m E := \m G \ominus \overline{span}_{j =1,\ldots, d} V_j \m G$.
\end{itemize}

Our goal is to find an outgoing Cuntz scattering system inside our model.
Let as before $\uu E$ be a coisometric lifting of a row contraction $\uu C$
by $\uu A$ and $\uu {\widehat V}^E$ be the minimal isometric dilation of $\uu E$
of the form given by equation (\ref{2}).
First we show that tuples $ \underline {\widehat  V}^E = (\widehat V_1^E,\ldots
,\widehat V_d^E)$
 and $\uu V^E = (V_1^E, \ldots, V_d^E)$ on Hilbert spaces $\widehat{\m H}_E$
and $\m H_A \oplus (\Gamma \otimes \m D_E)$ are row unitary and row isometry respectively.
 Since $ \underline {\widehat  V}^E = (\widehat V_1^E,\ldots
,\widehat V_d^E)$ is the minimal
 isometric dilation of $\uu E = (E_1, \ldots, E_d)$, it follows that
$\widehat V_j^E$'s are isometries with orthogonal ranges. So $V_j^E$'s are isometries with orthogonal ranges, because
\[
V_j^E = \widehat V_j^E|_{\m H_A \oplus (\Gamma \otimes \m D_E)} : \m H_A \oplus (\Gamma \otimes \m D_E) \to
\m H_A \oplus (\Gamma \otimes \m D_E)
 ~\mbox{for}~ j= 1, \ldots, d.
\]
Thus $\uu V^E = (V_1^E, \ldots, V_d^E)$ is a row isometry.
Also $\sum_{j=1}^d \widehat V_j^E (\widehat V_j^E)^* = I,$ i.e.,
$ \underline {\widehat V}^E = (\widehat V_1^E,\ldots
,\widehat V_d^E)$ is a row unitary.
Define  $\m E_* := W^*  (e_\emptyset \otimes \m D_C).$ We claim that $\m E_*$ is a
wandering subspace with respect to $\uu V^E$.
 It is enough to prove that
\[
 W^*( e_\emptyset \otimes \m D_C) \perp \overline{span}_{j=1,\ldots, d}V_j^E(\m H_A \oplus (\Gamma \otimes \m D_E)),
\]
since $V_j^E$'s are isometries with orthogonal ranges.  If $y\in \m D_C, h_a \in \m H_A$, and
$  \sum_{\alpha \in \tilde{\Lambda}} e_\alpha \otimes \eta_\alpha\in \Gamma \otimes \m D_E $, then
\begin{eqnarray*}
&&\langle~ W^* (e_\emptyset \otimes y), V_j^E(h_a \oplus  \sum_{\alpha \in
\tilde{\Lambda}} e_\alpha \otimes \eta_\alpha)~\rangle\\
 & = & \langle ~U_1^*\tilde {U_1}(e_\emptyset \otimes y) , E_j h_a \oplus e_\emptyset \otimes (D_E)_jh_a
\oplus  \sum_{\alpha \in \tilde{\Lambda}} e_j \otimes e_\alpha \otimes \eta_\alpha  ~\rangle \\
& = & \langle ~\sum_{i = 1}^d \big((D_C)_i^* y_\emptyset \otimes \epsilon_i\big), U_1  \big(E_j h_a
 \oplus e_\emptyset \otimes (D_E)_jh_a
\oplus  \sum_{\alpha \in \tilde{\Lambda}} e_j \otimes e_\alpha \otimes \eta_\alpha\big)~\rangle\\
& = & \langle ~\sum_{i = 1}^d \big((D_C)_i^* y_\emptyset \otimes \epsilon_i\big) ,
\sum_{i = 1}^d \big((E_i^*E_jh_a +(D_E)_i^*(D_E)_jh_a)
 \otimes \epsilon_i\big) \oplus  \sum_{\alpha \in \tilde{\Lambda}} e_j \otimes e_\alpha \otimes \eta_\alpha  ~\rangle\\
&=& \langle ~\sum_{i = 1}^d\big( (D_C)_i^* y_\emptyset \otimes \epsilon_i\big) , h_a
 \otimes \epsilon_j \oplus  \sum_{\alpha \in \tilde{\Lambda}} e_j \otimes e_\alpha \otimes \eta_\alpha  ~\rangle
=0,
\end{eqnarray*}
for $j=1,\ldots,d.$ The last equality holds because
$(D_C)_i^*y_\emptyset \in \m H_C$ for  $i= 1,\ldots, d$
 and $h_a \in \m H_A$. So our claim is established.

Our next aim is to prove that
\[
 \m E_* = \big(\m H_A \oplus (\Gamma \otimes \m D_E)\big) \ominus
 \overline{span}_{j=1,\ldots,d}V_j^E\big(\m H_A \oplus (\Gamma \otimes \m D_E)\big).
\]
The proof of the above claim shows that
\begin{equation}\label{3.1}
 \m E_* = W^*  (e_\emptyset \otimes \m D_C) \subset \big(\m H_A \oplus (\Gamma \otimes \m D_E)\big) \ominus
 \overline{span}_{j=1,\ldots,d}V_j^E\big(\m H_A \oplus (\Gamma \otimes \m D_E)\big).
\end{equation}
To show the reverse inclusion, let
$$x \in  \big(\m H_A \oplus (\Gamma \otimes \m D_E)\big) \ominus
 \overline{span}_{j=1,\ldots,d}V_j^E\big(\m H_A \oplus (\Gamma \otimes \m D_E)\big).$$
We can write $x = u \oplus v$ where $u \in W^*(e_\emptyset \otimes \m D_C)$ and
$v \in \big(\m H_A \oplus (\Gamma \otimes \m D_E)\big) \ominus W^*(e_\emptyset \otimes \m D_C).$
It follows from equation (\ref{3.1}) that
$$u \in  \big(\m H_A \oplus (\Gamma \otimes \m D_E)\big) \ominus
 \overline{span}_{j=1,\ldots,d}V_j^E\big(\m H_A \oplus (\Gamma \otimes \m D_E)\big).$$
Then
$v = x-u  \in \big(\m H_A \oplus (\Gamma \otimes \m D_E)\big) \ominus
 \overline{span}_{j=1,\ldots,d}V_j^E\big(\m H_A \oplus (\Gamma \otimes \m D_E)\big),$
i.e.,
\begin{equation}\label{3.2}
 v\perp \overline{span}_{j=1,\ldots,d}V_j^E\big(\m H_A \oplus
(\Gamma \otimes \m D_E)\big)  \big(= \overline{span}_{j=1,\ldots,d}\widehat
V_j^E\big(\m H_A \oplus (\Gamma \otimes \m D_E)\big)\big).
\end{equation}
Since $v\perp \m H_C =\widehat W^* \m H_C$ and
 $v\perp W^*(e_\emptyset \otimes \m D_C)  = \widehat W^* (e_\emptyset \otimes \m D_C)$, we conclude that
\begin{equation}\label{3.3a}
 v \perp \overline{span}_{j=1,\ldots,d}\widehat V_j^E\m H_C.
\end{equation}
By equations (\ref{3.2}) and (\ref{3.3a}), we see that
$v \perp \overline{span}_{j=1,\ldots,d}\widehat V_j^E\widehat{\m H}_E.$
We get $v\perp \widehat{\m H}_E$, since  $\widehat{\uu V}^E = (\widehat V_1^E, \ldots, \widehat V_d^E)$  is a
row unitary. Thus $v= 0$, and therefore $x = u \in W^*  (e_\emptyset \otimes \m D_C)$. This proves
the reverse inclusion.
Define $ \m G : = \Gamma \otimes \m D_E$.
Since $V_j^E|_{\m G } = L_j \otimes I_{\m D_E}$, it follows that $\uu V^E|_{\m G}= (V_1^E|_{\m G}, \ldots, V_d^E|_{\m G})$
is a row isometry, and $e_\emptyset \otimes \m D_E$ is a wandering subspace of $\m G$ with respect to $\uu V^E|_{\m G}$
such that
\[
 \m G = \bigoplus _{\alpha \in \tilde \Lambda}  V_\alpha^E ( e_\emptyset \otimes \m D_E).
\]
Thus $\uu V^E|_{\m G}= (V_1^E|_{\m G}, \ldots, V_d^E|_{\m G})  $ is a row shift.

We summarize the preceding discussion in the following:
\begin{theorem}
  A collection
\[
( \uu V^E = (V_1^E,\ldots, V_d^E), \m H_A \oplus (\Gamma \otimes \m D_E), \m G_{*}^+ =
 \bigoplus _{\alpha \in \tilde \Lambda}  V_\alpha^E \m E_*, \m G= \bigoplus _{\alpha \in \tilde \Lambda}  V_\alpha^E \m E )
\]
is an
outgoing Cuntz scattering system where $\m E_* = W^*(e_\emptyset \otimes \m D_C)$ and $\m E = e_\emptyset \otimes \m D_E$.
\end{theorem}

Take the input space as $\m D_E$ and the output space as $\m D_C$. If
\[
 \tilde C : = \sum_{j=1}^d (D_C)_j P_{\m H_C} E^*_j : \m H_E \to \m D_C ~~\mbox{and}
~~\tilde D := \sum_{j=1}^d (D_C)_j P_{\m H_C} (D_E)^*_j : \m D_E \to \m D_C
\]
where $P_{\m H_C}$ is the orthogonal projection of $\m H_E$ on to $\m H_C$,
 then we define a \textit{colligation of operators} \cite{BV05} as follows:

\[
 \m C_{U,\tilde U} :=  \begin{pmatrix}
          E_1^* & (D_E)_1^* \\
      \vdots & \vdots\\
        E_d^* & (D_E)_d^*\\
     \tilde C & \tilde D
         \end{pmatrix}
: \m H_E \oplus \m D_E \to \displaystyle\bigoplus_{j = 1}^d \m H_E \oplus \m D_C.
\]
Consider the following $\tilde \Lambda$-\textit{linear system} $\sum_{U, \tilde U}$
or
non-commutative Fornasini-Marchesini system in (cf. \cite{BGM06})
associated to the colligation $\m C_{U,\tilde U}$:
\begin{eqnarray*}
 x(j\alpha)  = E_j^* x(\alpha) + (D_E)_j^* u(\alpha)~~ \mbox{and}~
~ y(\alpha)   =   \tilde Cx(\alpha) + \tilde D u(\alpha)
\end{eqnarray*}
where $ j =  1, \ldots, d $ and $\alpha ,j\alpha$ are words in $ \tilde \Lambda$, and
\[
  x : \tilde \Lambda \to \m H_E,~~~ u : \tilde \Lambda \to \m D_E,~~~ y :\tilde \Lambda \to \m D_C.
\]

Let  $z= (z_1,\ldots, z_d)$  be a d-tuple of formal non-commuting indeterminates.
Define Fourier transforms of $x$, $u$ and $y$ as

 \[
 \hat x(z) = \displaystyle \sum_{\alpha \in \tilde \Lambda} x(\alpha) z^{\alpha},~~
\hat u(z) =  \sum_{\alpha \in \tilde \Lambda} u(\alpha) z^{\alpha},~
~ \hat y(z) = \sum_{\alpha
 \in \tilde \Lambda} y(\alpha) z^{\alpha}
\]
respectively
where  $z^\alpha=z_{\alpha_n}\ldots  z_{\alpha_1}$ for
$\alpha = \alpha_n\ldots \alpha_1 \in \tilde \Lambda$.
If we assume that $x(\emptyset) = 0$ and $z$-variables commute with the coefficients, then we get the
input-output relation
\[
 \hat y(z) = \Theta_{U,\tilde U} (z) \hat u(z)
\]
where $\Theta_{U,\tilde U}$ as a formal non-commutative  power series is given by the following:
\begin{equation*}
 \Theta_{U,\tilde U}(z) := \displaystyle \sum_{\alpha \in \tilde \Lambda} \Theta_{U,\tilde U}^{(\alpha)}
z^\alpha := \tilde D + \tilde C  \displaystyle \sum_{\substack {\beta \in \tilde
\Lambda \\ j = 1,\ldots ,d}}(E_{\bar\beta})^* (D_E)_j^* z^{\beta j}.
 \end{equation*}
Here $\bar \beta = \beta_1\ldots \beta_n$  is the reverse of $\beta  = \beta_n \ldots \beta_1$ and
$\Theta_{U,\tilde U}^{(\alpha)}$ are operators from $\m D_E$ to $\m D_C$.
We refer $\Theta_{U,\tilde U}$  as \textit{transfer function} associated to the unitaries $U$ and $\tilde U$.
For a Hilbert space $\m V$, a non-commutative analogue of Hardy space is
the space $ \ell^2(\tilde \Lambda , \m V)$ of formal power series
$g(z) =  \sum_{\alpha \in \tilde
\Lambda} g_\alpha z^\alpha$ with $\| g\|^2_{\ell^2} =  \sum_{\alpha \in \tilde
\Lambda}\| g_\alpha\|^2 < \infty$  where $g_\alpha \in \m V.$ The following theorem shows
that the formal non-commutative power series  $\Theta_{U,\tilde U}$
turns out to be a contractive operator between Hilbert spaces.

\begin{theorem}\label{3.7}
 The map $M_{\Theta_{U, \tilde U}} : \ell^2(\tilde \Lambda , \m D_E) \to \ell^2(\tilde \Lambda , \m D_C)$ defined by
\[
 M_{\Theta_{U, \tilde U}} \hat u(z) :=  \Theta_{U,\tilde U}(z) \hat u(z)
\]
is a contraction .
\end{theorem}

$M_{\Theta_{U, \tilde U}}$ intertwines with right translation i.e.,
\[
 M_{\Theta_{U, \tilde U}}(\sum_{\alpha \in \tilde \Lambda} u(\alpha) z^\alpha z^i) =
M_{\Theta_{U, \tilde U}}(\sum_{\alpha \in \tilde \Lambda}u(\alpha) z^\alpha) z^i
\]
for $i=1,\ldots,d$. Thus $M_{\Theta_{U, \tilde U}}$ is a \textit{multi-analytic operator} (cf. \cite{Po95}).
Since $M_{\Theta_{U, \tilde U}}$
is a contractive operator, the transfer function $\Theta_{U, \tilde U} \in  \m S_{nc,d}(\m D_E, \m D_C) $
  (non-commutative $d$-variable Schur class, cf. section 2.4 of \cite{BV05}) where
\[
 \m S_{nc,d}(\m D_E, \m D_C) := \{ T(z) =  \displaystyle \sum_{\alpha \in \tilde \Lambda} T_\alpha z^{\alpha}:
M_T : \ell^2( \tilde \Lambda, \m D_E) \to \ell^2( \tilde \Lambda, \m D_C) ~\mbox{satisfies}~ \| M_T \| \leq 1\}.
\]

Next we show that the transfer function 
coincides with the characteristic function of lifting \cite{DG11}.
Define unitaries $\Psi_C: \ell^2(\tilde \Lambda, \m D_C) \to \Gamma \otimes \m D_C$
 and
\newline $\Psi_E : \m D_E z ^\emptyset \to e_\emptyset\otimes \m D_E $ by
\[
 \Psi_C \big( \sum_{\alpha \in \tilde \Lambda} y_\alpha z^{\alpha}\big) =
\sum_{\alpha \in \tilde \Lambda } e_{\bar\alpha}\otimes y_\alpha
~~\mbox{and}~~
\Psi_E  (\eta z^\emptyset) = e_\emptyset \otimes \eta
\]
respectively where $ y_\alpha \in \m D_C$ and $\eta \in \m D_E$.
We observe that $\tilde C$ vanishes on $\m H_C$ by the following argument. For $\tilde h \in \m H_C$ we have
\begin{eqnarray}
 \tilde C \tilde h = \sum_{j=1}^d (D_C)_j P_{\m H_C} E^*_j \tilde h = \sum_{j=1}^d (D_C)_j P_{\m H_C} C^*_j \tilde h
= \sum_{j=1}^d (D_C)_j  C^*_j \tilde h
 = D_C \uu C^* \tilde h.   \nonumber
\end{eqnarray}
Since $\uu C$ is a coisometric tuple, the operator
 $D_C$ is the orthogonal projection. So we have
\begin{eqnarray}
 \tilde C \tilde h = D_C^2\uu C^* \tilde h
= (I - \uu C^* \uu C)\uu C^* \tilde h
 =(\uu C^* - ~\uu C^* \uu C~ \uu C^*) \tilde h
 =(\uu C^* - \uu C^*)\tilde h
 =  0. \label{3.3}
\end{eqnarray}
The second last equality follows by $\uu C$ is a coisometric tuple. Further, for $h_a \in \m H_A$
\begin{eqnarray}
\tilde C h_a &= & \sum_{j=1}^d (D_C)_j P_{\m H_C} E^*_j h_a = \sum_{j=1}^d (D_C)_j P_{\m H_C} 
(B^*_j h_a\oplus A^*_j h_a)\nonumber\\
&= &\sum_{j=1}^d (D_C)_j B^*_j h_a =  D_C \uu B^* h_a=  \uu B^* h_a.\label{3.4}
\end{eqnarray}
The last equality holds because  $D_C$ is the orthogonal projection and Range $\uu B^* \subset \m D_C$.
Define $D_{*,A} := (I -\uu A\uu A^* )^ 2 : \m H_A \to\m H_A$ and $\m D_{*,A} : = \overline {\mbox{Range} ~ D_{*,A}}.$
Since $\uu E$ is a coisometric lifting of
$\uu C$,
it follows from Theorem 2.1 of \cite{DG11} that there exist an isometry $\gamma : \m D_{*,A} \to
\m D_C$ with $\gamma D_{*,A} h_a = \uu B ^* h_a$ for  each $h_a \in \m  H_A$. By equation (\ref{3.4}), we have
\begin{eqnarray}\label{3.5}
 \tilde C h_a  =   \gamma D_{*,A} h_a ~\mbox{ for  each}~ h_a \in \m  H_A.
\end{eqnarray}

We recall the following expansion of the symbol of the characteristic function $M_{C, E}: \Gamma \otimes \m D_E
\to \Gamma \otimes \m D_C$
 of lifting $\uu E$ of $\uu C$ from \cite{DG11}: For
$h  \in \m H_C$
\begin{eqnarray}\label{4.1}
 \Theta_{C, E} (D_E)_ih = e_\emptyset \otimes [(D_C)_ih - \gamma D_{*, A}B_i h] - \displaystyle
\sum_{ |\alpha | \geq 1} e_\alpha \otimes \gamma D_{*, A} (A_{\alpha})^* B_ih,
\end{eqnarray}
and for $h \in \m H_A$
 \begin{eqnarray}\label{4.2}
 \Theta_{C, E} (D_E)_ih  =  - e_\emptyset \otimes  \gamma D_{*, A}A_i h + \sum_{ j = 1,\ldots ,d}
e_j \otimes \displaystyle \sum_{
 \alpha \in \tilde \Lambda} e_\alpha  \otimes \gamma D_{*, A}
(A_{\alpha})^*(\delta_{ij}I- A_j^*A_i) h
 \end{eqnarray}
where $i =1, \ldots, d$.

\begin{theorem}
  The transfer function $\Theta_{U, \tilde U}$ and the characteristic function $ \Theta_{C, E}$
 are related by the formula
\[
\Psi_C \Theta_{U,\tilde U} (z) =\Theta_{C, E} \Psi_E.
\]
In other words,  the transfer function $\Theta_{U, \tilde U}$ coincides with the 
characteristic function $ \Theta_{C, E}$.

\end{theorem}
\begin{proof}
 Let $h \in \m H_E$. For $i = 1, \ldots, d$
\begin{eqnarray}\label{3.8}
\Psi_C \Theta_{U, \tilde U}( (D_E)_i  h z^\emptyset)
\nonumber & = & \Psi_C  [\tilde D ~z^\emptyset+   \displaystyle
\sum_{\beta \in \tilde \Lambda , j = 1,\ldots ,d} \tilde C(E_{\bar\beta})^*
 (D_E)_j^* z^{\beta j}] ((D_E)_i  hz^\emptyset) \\
  & = & \Psi_C [\tilde D (D_E)_i  h~z^\emptyset+ \displaystyle
\sum_{\beta \in \tilde \Lambda , j = 1,\ldots ,d} \tilde C  (E_{\bar\beta})^* (D_E)_j^* (D_E)_i  h~z^{\beta j}].
\end{eqnarray}

Case 1. For $h \in \m H_C$ and $i=1,\ldots,d$
\begin{eqnarray}
 \tilde D  (D_E)_i  h
 &=& \displaystyle \sum_{j = 1}^d (D_C)_jP_{\m H_C}(D_E)_j^* (D_E)_i h
 = \displaystyle \sum_{j = 1}^d (D_C)_j P_{\m H_C}( \delta_{ij}I-E_j^*E_i) h \nonumber \\
& =& (D_C)_i h - \big (\displaystyle \sum_{j = 1}^d (D_C)_j P_{\m H_C} E_j^*\big)E_i h
 =  (D_C)_i h - \tilde C E_i h \nonumber\\
& = & (D_C)_i h - \tilde C (C_i h\oplus B_i h)
   =  (D_C)_i h-\tilde C B_i h. \label{3.9}
\end{eqnarray}
The last equality follows by equation (\ref{3.3}). The second part of equation (\ref{3.8}) simplifies to
\begin{eqnarray*}
 \displaystyle
\sum_{\beta \in \tilde \Lambda , j = 1,\ldots ,d} \tilde C  (E_{\bar\beta})^* (D_E)_j^* (D_E)_i  h~z^{\beta j} \nonumber
& = &
\displaystyle   \sum_{\beta \in \tilde \Lambda, j = 1,\ldots ,d}
\tilde C (E_{\bar\beta})^*( \delta_{ij}I-E_j^*E_i) h ~z^{\beta j} \\
\nonumber&=& \displaystyle \sum_{\beta \in \tilde \Lambda}
\tilde C (E_{\bar\beta})^* h~z^{\beta i}  - \displaystyle   \sum_{\beta \in \tilde \Lambda ,j = 1,\ldots ,d}
\tilde C (E_{\bar\beta})^* E_j^*E_i h ~z^{\beta j}.
\end{eqnarray*}
By equation (\ref{3.3}) it follows that
\begin{eqnarray}
 \displaystyle
\sum_{\beta \in \tilde \Lambda , j = 1,\ldots ,d} \tilde C  (E_{\bar\beta})^* (D_E)_j^* (D_E)_i  h~z^{\beta j}
\nonumber&=& - \displaystyle   \sum_{\beta \in \tilde \Lambda, j = 1,\ldots ,d}
\tilde C (E_{\bar\beta})^* E_j^*E_i h ~z^{\beta j}\\
\nonumber&=& - \displaystyle  \sum_{\beta \in \tilde \Lambda, j = 1,\ldots ,d}
\tilde C (E_{\bar\beta})^* \big((C_j^*C_i +B_j^*B_i)h \oplus A_j^*B_ih\big) ~z^{\beta j}\\
\nonumber &=&  - \displaystyle    \sum_{\beta \in \tilde \Lambda ,j = 1,\ldots ,d}
\tilde C (A_{\bar\beta})^*A_j^*B_ih ~z^{\beta j}\\
&=&- \displaystyle   \sum_{ |\alpha | \geq 1}
 \tilde C (A_{\bar\alpha})^* B_ih ~z^{\alpha}.\label{3.10}
\end{eqnarray}
The  equality which is second from below in the above equation array follows by equation (\ref{3.3}).
By equations (\ref{3.8}), (\ref{3.9}), and (\ref{3.10}), we have for $i= 1, \ldots,d$ and $h \in \m H_C$
\begin{eqnarray}
\Psi_C \Theta_{U, \tilde U} ((D_E)_i h z^\emptyset)
\nonumber &=& \Psi_C[(D_C)_i h-\tilde C B_i h)~z^\emptyset -
\displaystyle   \sum_{ |\alpha | \geq 1} \tilde C (A_{\bar\alpha})^* B_ih ~z^{\alpha}]\\
\nonumber &=& e_\emptyset \otimes  ((D_C)_i h-\tilde C B_i h) - \displaystyle
\sum_{|\alpha | \geq 1}
e_{\bar\alpha} \otimes \tilde C (A_{\bar\alpha})^* B_ih \\
\nonumber &=& e_\emptyset \otimes [(D_C)_ih - \gamma D_{*, A}B_i h] - \displaystyle   \sum_{
 |\alpha | \geq 1} e_{\bar\alpha} \otimes \gamma D_{*, A} (A_{\bar\alpha})^* B_ih.
\end{eqnarray}
By equation (\ref{4.1}) we obtain
 \begin{eqnarray}
\Psi_C \Theta_{U, \tilde U} ((D_E)_i h z^\emptyset)
 \nonumber &=&  \Theta_{C, E} (e_\emptyset \otimes (D_E)_i h)\\
&=& \Theta_{C, E} \Psi_E ((D_E)_i h z^\emptyset). \label{3.11}
 \end{eqnarray}
Case 2. For $h \in \m H_A$ and $i= 1, \ldots,d$
\begin{eqnarray}
 \tilde D  (D_E)_i  h
 \nonumber &=& \displaystyle \sum_{j = 1}^d (D_C)_j P_{\m H_C}(D_E)^*_j (D_E)_i h
= \displaystyle \sum_{j = 1}^d (D_C)_j P_{\m H_C}( \delta_{ij}I-E_j^*E_i) h\\
 &=& (D_C)_i P_{\tilde{\m H}}  h - \big (\displaystyle \sum_{j = 1}^d D_j P_{\m H_C} E_j^*\big)E_i h
=  - \tilde C A_i h.     \label{3.12}
  \end{eqnarray}
Consider again the second part of equation (\ref{3.8}).
\begin{eqnarray}\label{3.13}
 \displaystyle \sum_{\beta \in \tilde \Lambda ,j = 1,\ldots ,d}\tilde C (E_{\bar\beta})^*
(D_E)_j^* (D_E)_i  h~z^{\beta j}
  \nonumber& = &
\displaystyle   \sum_{\beta \in \tilde \Lambda,j = 1,\ldots ,d}
\tilde C (E_{\bar\beta})^*( \delta_{ij}I-E_j^*E_i) h ~z^{\beta j} \\
&=&  \displaystyle   \sum_{\beta \in \tilde \Lambda,
 j = 1,\ldots ,d}\tilde C (A_{\bar\beta})^*(\delta_{ij}I- A_j^*A_i) h~z^{\beta j}.   \label{3.13}
\end{eqnarray}
The last equality follows from equation (\ref{3.3}).
By equations (\ref{3.8}), (\ref{3.12}), and (\ref{3.13}), we have for $i= 1, \ldots, d$ and  $h\in \m H_A$
\begin{eqnarray}
 \Psi_C \Theta_{U, \tilde U} ((D_E)_i h z^\emptyset)
\nonumber &=& \Psi_C [ - \tilde C A_i h ~z^\emptyset +  \displaystyle   \sum_{\beta \in \tilde
\Lambda, j = 1,\ldots ,d}\tilde C (A_{\bar\beta})^*(\delta_{ij}I- A_j^*A_i) h~z^{\beta j}]\\
\nonumber &=& - e_\emptyset \otimes \tilde C A_i h +
\displaystyle   \sum_{\beta \in \tilde \Lambda , j = 1,\ldots ,d} e_j \otimes
e_{\bar\beta}  \otimes  \tilde C (A_{\bar\beta})^*(\delta_{ij}I- A_j^*A_i) h)\\
\nonumber &=& - e_\emptyset \otimes  \gamma D_{*, A}A_i h + \displaystyle
\sum_{\beta \in \tilde \Lambda, j = 1,\ldots ,d}
e_j \otimes e_{\bar\beta}   \otimes \gamma D_{*, A} (A_{\bar\beta})^*(\delta_{ij}I- A_j^*A_i) h.
\end{eqnarray}
Equation (\ref{4.2}) yields
 \begin{eqnarray}
 \Psi_C \Theta_{U, \tilde U} ((D_E)_i h z^\emptyset) &=&  \Theta_{C, E} (e_\emptyset \otimes (D_E)_ih
\nonumber\\
&=& \Theta_{C, E} \Psi_E( (D_E)_i h z^\emptyset). \label{3.15}
\end{eqnarray}
We infer from equations (\ref{3.11}) and (\ref{3.15}) that
\[
 \Psi_C \Theta_{U, \tilde U}= \Theta_{C,E} \Psi_E.
\]\end{proof}

We extend  the unitary $\Psi_E$ to a unitary (also denoted by $\Psi_E$) 
from $ \m H_A \oplus \ell^2(\tilde \Lambda, \m D_E)$
on to $\m H_A \oplus (\Gamma \otimes \m D_E)$ by
\[
\Psi_E(h_a \oplus \sum_{\alpha \in \tilde \Lambda} \eta_\alpha z^{\alpha})=
(h_a\oplus \sum_{\alpha \in \tilde \Lambda } e_{\bar\alpha} \otimes \eta_\alpha) 
\]
where $h_a \in \m H_A$ and $\eta_\alpha \in \m D_E$.
Using the unitaries  $\Psi_C,\Psi_E$, and the coisometry $W$ of section
 2 we define $\Psi_W$ by the following commutative
diagram:

\begin{equation}
\xymatrix{
\m H_A \oplus (\Gamma \otimes \m D_E) \ar[r]^{{ W}} \ar[d]_{\Psi_E^{-1}}
&   \Gamma \otimes \m D_C \ar[d]^{\Psi_C^{-1}}
\\
\m H_A \oplus \ell^2 (\tilde \Lambda, \m D_E) \ar[r]^{\Psi_W}
&   \ell^2 (\tilde \Lambda, \m D_C),
}
\end{equation}
i.e., $\Psi_W  = \Psi_C^{-1} W \Psi_E$. Similar to Theorem 5.1 of \cite{DH11} we have

\begin{theorem}\label{b}
  The operator $\Psi_W$ satisfies the relation
\[
 \Psi_W|_{\ell^2 (\tilde \Lambda, \m D_E)} = M_{\Theta_{U, \tilde U}} .
\]
\end{theorem}
\noindent Observe that we also obtain 
\[
 W|_{e_\emptyset \otimes \m D_E} = \Theta_{C, E}.
\]


\noindent {\bf Acknowledgement:} The author is thankful to Santanu Dey 
 for many helpful discussions.

\end{document}